\documentclass[12pt]{amsart}
\usepackage{amsmath,amsthm,amsfonts,amscd,amssymb,eucal,latexsym,mathrsfs, appendix, yhmath, setspace}
\usepackage[numbers,sort&compress]{natbib}
\usepackage[all,cmtip]{xy}
\usepackage{enumerate}
\newtheorem{theorem}{Theorem}[section]

\newtheorem{proposition}[theorem]{Proposition}

\theoremstyle{definition}
\newtheorem{definition}[theorem]{Definition}

\allowdisplaybreaks
\begin{document}

\title{Hypercyclic and mixing composition operators on $H^{p}$}

\author{Zhen Rong}

\address{\hskip-\parindent
Z.R., College of statistics and mathematics, Inner Mongolia University of Finance and Economics,
Hohhot 010000, China.}
\email{rongzhenboshi@sina.com}

\date{June 1, 2023}


\begin{abstract}
Extending previous results of Bourdon and Shapiro we characterize the hypercyclic and mixing composition operators $C_{\varphi}$ for the automorphisms of $\mathbb{D}$ on any of the spaces $H^{p}$ with $1\leqslant p<+\infty$.
\end{abstract}

\maketitle

\section{Introduction} \label{S-introduction}
Throughout this article, let $\mathbb{N}$ denote the set of nonnegative integers. Let $\mathbb{C}$ denote the complex number field. Let $\mathbb{T}=\{z\in\mathbb{C}:|z|=1\}$ and $\mathbb{D}=\{z\in\mathbb{C}:|z|<1\}$.

An {\it automorphism} of $\mathbb{D}$ is a bijective analytic function $\varphi:\mathbb{D}\rightarrow\mathbb{D}$. The set of all automorphisms of $\mathbb{D}$ is denoted by $Aut(\mathbb{D})$. It is well known that the automorphisms of $\mathbb{D}$ are the linear fractional transformations
$$\varphi(z)=b\frac{z-a}{1-\overline{a}z}, |a|<1, |b|=1.$$
Moreover, every $\varphi\in Aut(\mathbb{D})$ maps $\mathbb{T}$ bijectively onto itself (see \cite[pages 131-132]{Conway}).

For $1\leqslant p<+\infty$, let $H^{p}$ denote the space of all analytic functions on $\mathbb{D}$ for which
$$\sup\limits_{0\leqslant r<1}(\frac{1}{2\pi}\int_{0}^{2\pi}|f(re^{i\theta})|^{p}d\theta)^{\frac{1}{p}}<+\infty.$$
For any $f\in H^{p}$, let
$$\|f\|_{p}=\sup\limits_{0\leqslant r<1}(\frac{1}{2\pi}\int_{0}^{2\pi}|f(re^{i\theta})|^{p}d\theta)^{\frac{1}{p}}.$$
Then $(H^{p},\|\cdot\|_{p})$ is a Banach space.

Let $\varphi\in Aut(\mathbb{D})$ and let $C_{\varphi}(f)=f\circ\varphi(f\in H^{p})$ be the corresponding composition operator on $H^{p}$. It is well known that for any $1\leqslant p<+\infty$ and $\varphi\in Aut(\mathbb{D})$, $C_{\varphi}$ defines a continuous linear operator on $H^{p}$ (see \cite[pages 220-221]{Zhu}).

A continuous linear operator $T$ on a Banach space $X$ is called {\it hypercyclic} if there is an element $x$ in $X$ whose orbit $\{T^{n}x:n\in\mathbb{N}\}$ under $T$ is dense in $X$; {\it topologically transitive} if for any pair $U,V$ of nonempty open subsets of $X$, there exists some nonnegative integer $n$ such that $T^{n}(U)\cap V\neq\emptyset$; and {\it mixing} if for any pair $U,V$ of nonempty open subsets of $X$, there exists some nonnegative integer $N$ such that $T^{n}(U)\cap V\neq\emptyset$ for all $n\geqslant N$.

Recently Bourdon and Shapiro \cite{Bourdon-Shapiro1,Bourdon-Shapiro2} have done an extensive study of cyclic and hypercyclic linear fractional composition operators on $H^{2}$. Zorboska \cite{Zorboska} has determined hypercyclic and cyclic composition operators induced by a linear fractional self map of $\mathbb{D}$, acting on a special class of smooth weighted Hardy spaces $H^{2}(\beta)$. Gallardo and Montes \cite{Gallardo-Montes} have obtained a complete characterization of the cyclic, supercyclic and hypercyclic composition operators $C_{\varphi}$ for linear fractional self-maps $\varphi$ of $\mathbb{D}$ on any of the spaces $\mathcal{S}_{\nu},\nu\in\mathbb{R}$. In particular, $\mathcal{S}_{0}$ is the Hardy space $H^{2}$, $\mathcal{S}_{-1/2}$ is the Bergman space $A^{2}$, and $\mathcal{S}_{1/2}$ is the Dirichlet space $\mathcal{D}$ under an equivalent norm. Since the Hardy space $H^{2}$ is a particular case of the spaces $H^{p}$ with $1\leqslant p<+\infty$, it is therefore very natural to try to characterize the cyclic, supercyclic and hypercyclic composition operators $C_{\varphi}$ for linear fractional self-maps $\varphi$ of $\mathbb{D}$ on any of the spaces $H^{p}$ with $1\leqslant p<+\infty$. In this paper we will characterize the hypercyclic and mixing composition operators $C_{\varphi}$ for the automorphisms of $\mathbb{D}$ on any of the spaces $H^{p}$ with $1\leqslant p<+\infty$, generalizing the corresponding results in \cite{Bourdon-Shapiro1,Bourdon-Shapiro2}.

\begin{theorem}
Let $1\leqslant p<+\infty$. Let $\varphi\in Aut(\mathbb{D})$ and $C_{\varphi}$ be the corresponding composition operator on $H^{p}$. Then the following assertions are equivalent:

(1) $C_{\varphi}$ is hypercyclic;

(2) $C_{\varphi}$ is mixing;

(3) $\varphi$ has no fixed point in $\mathbb{D}$.
\end{theorem}

Bourdon and Shapiro \cite{Bourdon-Shapiro1,Bourdon-Shapiro2} proved the above result in the case $p=2$. Hence the above result generalizes the corresponding results in \cite{Bourdon-Shapiro1,Bourdon-Shapiro2}.

This paper is organized as follows. In Section~\ref{S-hypercyclic} we characterize the hypercyclic and mixing composition operators $C_{\varphi}$ for the automorphisms of $\mathbb{D}$ on any of the spaces $H^{p}$ with $1\leqslant p<+\infty$.

\section{Hypercyclic and mixing composition operators on $H^{p}$} \label{S-hypercyclic}
In this section we characterize the hypercyclic and mixing composition operators $C_{\varphi}$ for the automorphisms of $\mathbb{D}$ on any of the spaces $H^{p}$ with $1\leqslant p<+\infty$, generalizing the corresponding results in \cite{Bourdon-Shapiro1,Bourdon-Shapiro2}.

The following propositions are the major techniques we need.

If $f\in H^{p}(1\leqslant p<+\infty)$, then $\lim\limits_{r\rightarrow1^{-}}f(re^{i\theta})$ exists for almost all values of $\theta$ (see \cite[page 17]{Duren}), thus defining a function which we denote by $f(e^{i\theta})$.

We need the following important properties of the $H^{p}$-spaces (see \cite[pages 9, 12, 21]{Duren}).

\begin{proposition}
Let $f\in H^{p},1\leqslant p<+\infty$. Then

(1) $\|f\|_{p}=\lim\limits_{r\rightarrow1^{-}}(\frac{1}{2\pi}\int_{0}^{2\pi}|f(re^{i\theta})|^{p}d\theta)^{\frac{1}{p}}$;

(2) $\lim\limits_{r\rightarrow1^{-}}\int_{0}^{2\pi}|f(re^{i\theta})|^{p}d\theta=\int_{0}^{2\pi}|f(e^{i\theta})|^{p}d\theta$;

(3) $\lim\limits_{r\rightarrow1^{-}}\int_{0}^{2\pi}|f(re^{i\theta})-f(e^{i\theta})|^{p}d\theta=0$.
\end{proposition}

We need the following property of the Taylor coefficients of $H^{p}$ functions (see \cite[page 94]{Duren}).

\begin{proposition}
Let $1\leqslant p\leqslant2$ and $f\in H^{p}$. Let $\sum\limits_{n=0}^{\infty}a_{n}z^{n}$ be the Taylor series of $f$ at the origin. Then $\{a_{n}\}_{n=0}^{\infty}\in l^{q}$, where $\frac{1}{p}+\frac{1}{q}=1$.
\end{proposition}

Our aim now is to characterize when $C_{\varphi}$ is hypercyclic on $H^{p}$. To this end we need some important dynamical properties of automorphisms $\varphi$ of $\mathbb{D}$.

Let
$$\varphi(z)=\frac{az+b}{cz+d}, ad-bc\neq0$$
be an arbitrary linear fractional transformation, which we consider as a map on the extended complex plane $\widehat{\mathbb{C}}$. Then $\varphi$ has either one or two fixed points in $\widehat{\mathbb{C}}$, or it is the identity.

Suppose that $\varphi$ has two distinct fixed points $z_{0}$ and $z_{1}$, and let $\sigma$ be a linear fractional transformation that maps $z_{0}$ to 0 and $z_{1}$ to $\infty$. Then $\psi:=\sigma\circ\varphi\circ\sigma^{-1}$ has fixed points 0 and $\infty$, which easily implies that $\psi(z)=\lambda z$ for some $\lambda\neq0$. The constant $\lambda$ is called the \emph{multiplier} of $\varphi$. Replacing $\sigma$ by $1/\sigma$ one sees that also $1/\lambda$ is a multiplier, which, however, causes no problem in the following.

\begin{definition}
Let $\varphi$ be a linear fractional transformation that is not the identity.

(a) If $\varphi$ has a single fixed point then it is called \emph{parabolic}.

(b) Suppose that $\varphi$ has two distinct fixed points, and let $\lambda$ be its multiplier. If $|\lambda|=1$ then $\varphi$ is called \emph{elliptic}; if $\lambda>0$ then $\varphi$ is called \emph{hyperbolic}; in all other cases, $\varphi$ is called \emph{loxodromic}.
\end{definition}

We need the following dynamical properties of automorphisms $\varphi$ of $\mathbb{D}$ (see \cite[pages 125-126]{Grosse-Erdmann-Peris}).

\begin{proposition}
Let $\varphi\in Aut(\mathbb{D})$, not the identity. Then we have the following:

(i) if $\varphi$ is parabolic then its fixed point $z_{0}$ lies in $\mathbb{T}$, and $\varphi^{n}(z)\rightarrow z_{0}, \varphi^{-n}(z)\rightarrow z_{0}$ for all $z\in\widehat{\mathbb{C}}$;

(ii) if $\varphi$ is elliptic then it has a fixed point in $\mathbb{D}$;

(iii) if $\varphi$ is hyperbolic then it has distinct fixed points $z_{0}$ and $z_{1}$ in $\mathbb{T}$ such that $\varphi^{n}(z)\rightarrow z_{0}$ for all $z\in\widehat{\mathbb{C}}, z\neq z_{1}$, and $\varphi^{-n}(z)\rightarrow z_{1}$ for all $z\in\widehat{\mathbb{C}}, z\neq z_{0}$;

(iv) $\varphi$ cannot be loxodromic.
\end{proposition}

The dynamical properties of $\varphi\in Aut(\mathbb{D})$ imply the dynamical properties of $C_{\varphi}$.

Finally we prove Theorem 1.1.

$\mathbf{Proof~of~Theorem~1.1.}$

(2)$\Rightarrow$(1) Assume that $C_{\varphi}$ is mixing. Then $C_{\varphi}$ is topologically transitive. Since the polynomials form a dense set in $H^{p}$, $H^{p}$
is separable. Since a continuous linear operator on a separable Banach space is topologically transitive if and only if it is hypercyclic (see \cite[page 10]{Grosse-Erdmann-Peris}), $C_{\varphi}$ is hypercyclic.

(1)$\Rightarrow$(3) Assume that $C_{\varphi}$ is hypercyclic. We will show that $\varphi$ has no fixed point in $\mathbb{D}$. Suppose $\varphi$ has a fixed point $z_{0}\in\mathbb{D}$. Since $C_{\varphi}$ is hypercyclic, there exists a $f\in H^{p}$ such that $\{(C_{\varphi})^{n}f:n\geqslant0\}$ is dense in $H^{p}$. We may choose a $g\in H^{p}$ with $g(z_{0})\neq f(z_{0})$. Since $\overline{\{(C_{\varphi})^{n}f:n\geqslant0\}}=H^{p}$, we may choose a sequence $\{n_{k}\}_{k=1}^{\infty}$ of positive integers with $n_{1}<n_{2}<\cdots<n_{k}<n_{k+1}<\cdots$ such that $\lim\limits_{k\rightarrow\infty}(C_{\varphi})^{n_{k}}f=g$. Since each point evaluation $k_{\lambda}:H^{p}\rightarrow\mathbb{C}(\lambda\in\mathbb{D})$ is continuous on $H^{p}$, where $k_{\lambda}(f)=f(\lambda)(f\in H^{p})$, we have $\lim\limits_{k\rightarrow\infty}((C_{\varphi})^{n_{k}}f)(z_{0})=g(z_{0})$. Notice that
$$((C_{\varphi})^{n_{k}}f)(z_{0})=(f\circ\varphi^{n_{k}})(z_{0})=f(\varphi^{n_{k}}(z_{0}))=f(z_{0}).$$
Hence $f(z_{0})=g(z_{0})$, this is a contradiction with $f(z_{0})\neq g(z_{0})$. Therefore $\varphi$ has no fixed point in $\mathbb{D}$.

(3)$\Rightarrow$(2) Suppose that $\varphi$ has no fixed point in $\mathbb{D}$. It suffices to show that $C_{\varphi}$ satisfies Kitai's criterion. By Proposition 2.4, $\varphi$ is either parabolic or hyperbolic, and in both cases $\varphi$ has fixed points $z_{0}$ and $z_{1}$ in $\mathbb{T}$ (possibly with $z_{0}=z_{1}$) such that $\varphi^{n}(z)\rightarrow z_{0}$ for all $z\in\mathbb{T}\backslash\{z_{1}\}$ and $\varphi^{-n}(z)\rightarrow z_{1}$ for all $z\in\mathbb{T}\backslash\{z_{0}\}$.

Now, for $X_{0}$ we will take the subspace of $H^{p}$ of all functions that are analytic on a neighbourhood of $\overline{\mathbb{D}}$ and that vanish at $z_{0}$. Since $z_{0}$ is a fixed point of $\varphi$, $C_{\varphi}$ maps $X_{0}$ into itself.

$\mathbf{Claim~1.}$ For any $1\leqslant p<+\infty$ we have $\overline{X_{0}}=H^{p}$.

We divide it into two cases.

Case i. If $1<p<+\infty$. First we will show that $X_{0}^{\bot}=\{0\}$. Let $1<p<+\infty$ and $x^{\ast}\in X_{0}^{\bot}$. We will show that $x^{\ast}=0$. Since $1<p<+\infty$ and $x^{\ast}\in(H^{p})^{\ast}$, there exists a unique function $g\in H^{q}$ such that
$$x^{\ast}(f)=\frac{1}{2\pi}\int_{0}^{2\pi}f(e^{i\theta})g(e^{-i\theta})d\theta(f\in H^{p}),$$
where $\frac{1}{p}+\frac{1}{q}=1$ (see \cite[pages 112-113]{Duren}). By Proposition 2.1 we have
$$\lim_{r\rightarrow1^{-}}\int_{0}^{2\pi}|g(re^{i\theta})-g(e^{i\theta})|^{q}d\theta=0.$$
Hence for each $f\in H^{p}$ we have
$$\lim_{r\rightarrow1^{-}}\frac{1}{2\pi}\int_{0}^{2\pi}f(e^{i\theta})g(re^{-i\theta})d\theta=\frac{1}{2\pi}\int_{0}^{2\pi}f(e^{i\theta})g(e^{-i\theta})d\theta.$$
Since, for any $n\geqslant0$, the functions $g_{n}:\mathbb{C}\rightarrow\mathbb{C}$ defined by $g_{n}(z)=z_{0}z^{n}-z^{n+1}$ belong to $X_{0}$ we have that $x^{\ast}(g_{n})=0(n\geqslant0)$. Notice that
\begin{align*}
x^{\ast}(g_{n})=&\frac{1}{2\pi}\int_{0}^{2\pi}g_{n}(e^{i\theta})g(e^{-i\theta})d\theta\\
=&\lim_{r\rightarrow1^{-}}\frac{1}{2\pi}\int_{0}^{2\pi}g_{n}(e^{i\theta})g(re^{-i\theta})d\theta.
\end{align*}
Hence for any $n\geqslant0$ we have
$$\lim_{r\rightarrow1^{-}}\frac{1}{2\pi}\int_{0}^{2\pi}g_{n}(e^{i\theta})g(re^{-i\theta})d\theta=0.$$
Let $g(z)=\sum\limits_{n=0}^{\infty}a_{n}z^{n}(z\in\mathbb{D})$, $0<r<1$ and $n\geqslant0$. Then
\begin{align*}
&\frac{1}{2\pi}\int_{0}^{2\pi}g_{n}(e^{i\theta})g(re^{-i\theta})d\theta\\
&=\frac{1}{2\pi}\int_{0}^{2\pi}(z_{0}e^{in\theta}-e^{i(n+1)\theta})g(re^{-i\theta})d\theta\\
&=\frac{1}{2\pi}\int_{0}^{2\pi}z_{0}e^{in\theta}g(re^{-i\theta})d\theta-\frac{1}{2\pi}\int_{0}^{2\pi}e^{i(n+1)\theta}g(re^{-i\theta})d\theta\\
&=\frac{1}{2\pi}\int_{0}^{2\pi}z_{0}e^{in\theta}(\sum_{k=0}^{\infty}a_{k}r^{k}e^{-ik\theta})d\theta-\frac{1}{2\pi}\int_{0}^{2\pi}e^{i(n+1)\theta}
(\sum_{k=0}^{\infty}a_{k}r^{k}e^{-ik\theta})d\theta\\
&=\sum_{k=0}^{\infty}\frac{1}{2\pi}\int_{0}^{2\pi}z_{0}e^{in\theta}a_{k}r^{k}e^{-ik\theta}d\theta-\sum_{k=0}^{\infty}\frac{1}{2\pi}\int_{0}^{2\pi}e^{i(n+1)\theta}
a_{k}r^{k}e^{-ik\theta}d\theta\\
&=\frac{1}{2\pi}\int_{0}^{2\pi}z_{0}a_{n}r^{n}d\theta-\frac{1}{2\pi}\int_{0}^{2\pi}a_{n+1}r^{n+1}d\theta\\
&=z_{0}a_{n}r^{n}-a_{n+1}r^{n+1}.
\end{align*}
Since $\lim\limits_{r\rightarrow1^{-}}\frac{1}{2\pi}\int_{0}^{2\pi}g_{n}(e^{i\theta})g(re^{-i\theta})d\theta=0$, we have $$\lim\limits_{r\rightarrow1^{-}}(z_{0}a_{n}r^{n}-a_{n+1}r^{n+1})=z_{0}a_{n}-a_{n+1}=0.$$
Hence $a_{n}=a_{0}z_{0}^{n}(n\geqslant0)$. Since $q=\frac{p}{p-1}>1$, we may choose $1<q_{1}\leqslant2$ with $q_{1}<q$. It is evident that $H^{q}\subseteq H^{q_{1}}$. Since $g\in H^{q}$, $g\in H^{q_{1}}$. By Proposition 2.2 we have $\{a_{n}\}_{n=0}^{\infty}\in l^{p_{1}}$, where $\frac{1}{p_{1}}+\frac{1}{q_{1}}=1$. Notice that $$\sum\limits_{n=0}^{\infty}|a_{n}|^{p_{1}}=\sum\limits_{n=0}^{\infty}|a_{0}z_{0}^{n}|^{p_{1}}=\sum\limits_{n=0}^{\infty}|a_{0}|^{p_{1}}<+\infty.$$
Hence $a_{0}=0$ and $a_{n}=0$ for $n\geqslant0$. Therefore $g(z)=0$ for all $|z|<1$ and $x^{\ast}=0$.

Second we will show that $\overline{X_{0}}=H^{p}$. Since $\overline{X_{0}}^{\bot}=X_{0}^{\bot}$ and $X_{0}^{\bot}=\{0\}$, we have $\overline{X_{0}}^{\bot}=\{0\}$. Hence $\overline{X_{0}}^{\bot\bot}=\{0\}^{\bot}=H^{p}$. Since $1<p<+\infty$, $H^{p}$ is reflexive. Since $\overline{X_{0}}$ is norm-closed, $\overline{X_{0}}$ is $\sigma(X^{\ast},X)$-closed. Finally we have $\overline{X_{0}}^{\bot\bot}=\overline{X_{0}}$. Hence $\overline{X_{0}}=H^{p}$. This proves the case $1<p<+\infty$.

Case ii. If $p=1$. We will show that $\overline{X_{0}}=H^{1}$. Let $f\in H^{1}$ and $\varepsilon>0$. We will show that there exists a $g\in X_{0}$ such that $\|f-g\|_{1}<\varepsilon$. Since the polynomials form a dense set in $H^{1}$, there exists a polynomial $h$ such that $\|f-h\|_{1}<\frac{\varepsilon}{2}$. By the case $p=2$, $X_{0}$ is dense in $H^{2}$. Hence there exists a $g\in X_{0}$ such that $\|g-h\|_{2}<\frac{\varepsilon}{2}$. Notice that $\|g-h\|_{1}\leqslant\|g-h\|_{2}$. Then $\|g-h\|_{1}<\frac{\varepsilon}{2}$. Hence $\|f-g\|_{1}\leqslant\|f-h\|_{1}+\|h-g\|_{1}<\varepsilon$. This proves the case $p=1$.

$\mathbf{Claim~2.}$ $(C_{\varphi})^{n}f\rightarrow0$ for all $f\in X_{0}$.

Let $f\in X_{0}$. Since $f\circ\varphi^{n}$ is continuous on $\overline{\mathbb{D}}$, by Proposition 2.1 we have
$$\|(C_{\varphi})^{n}f\|_{p}^{p}=\frac{1}{2\pi}\int_{0}^{2\pi}|f(\varphi^{n}(e^{i\theta}))|^{p}d\theta.$$
Since the integrands are uniformly bounded and convergent to $|f(z_{0}|^{p}=0$, for every $t$ with possibly one exception, the dominated convergence theorem implies that $(C_{\varphi})^{n}f\rightarrow0$. This proves Claim 2.

Next, for $Y_{0}$ we will take the subspace of $H^{p}$ of all functions that are analytic on a neighbourhood of $\overline{\mathbb{D}}$ and that vanish at $z_{1}$, and for $S$ we take the map $S=C_{\varphi^{-1}}$. Since $z_{1}$ is a fixed point of $\varphi^{-1}$, $S$ maps $Y_{0}$ into itself, and clearly $C_{\varphi}S=I$. It follows as above that $Y_{0}$ is dense in $H^{p}$ and that $S^{n}f\rightarrow0$ for all $f\in Y_{0}$. Therefore the conditions of Kitai's criterion are satisfied, so that $C_{\varphi}$ is mixing.

\begin{flushright}
  $\Box$
\end{flushright}

Bourdon and Shapiro \cite{Bourdon-Shapiro1,Bourdon-Shapiro2} proved Theorem 1.1 in the case $p=2$. Hence Theorem 1.1 generalizes the corresponding results in \cite{Bourdon-Shapiro1,Bourdon-Shapiro2}.


\end{document}